\documentclass[preprint,12pt]{elsarticle}

\usepackage{amsmath}               
\usepackage{amsthm}                
\usepackage{amssymb}
\usepackage{xcolor,colortbl}

\newtheorem{thm}{Theorem}[section]
\newtheorem{lem}[thm]{Lemma}
\newtheorem{prop}[thm]{Proposition}

\newtheorem{example}{Example}[section]

\newcommand{\ny}{\hspace*{-2.5mm}}

\newcommand{\hy}{\hspace*{2.5mm}}

\def\Diag{\textup{Diag}}
\newcommand{\tr}{\rm trace}

 \newcommand{\vs}{\vspace*{1mm}}
\newcommand{\ve}{\varepsilon}
\newcommand{\ds}{\displaystyle}

\newcommand{\R}{\mathbb{R}}      
\newcommand{\ZZ}{\mathbb{Z}}      
\newcommand{\cS}{\mathcal{S}}

\journal{Journal of Pure and Applied Algebra}

\begin{document}

\begin{frontmatter}

\title{On sums of squares of $k$-nomials \tnoteref{aaa}}

\author[1]{Jo\~ao Gouveia\corref{a}}
\ead{jgouveia@mat.uc.pt}
\author[1]{Alexander Kova\v cec}
\ead{kovacec@mat.uc.pt}
\author[1]{Mina Saee}
\ead{minasaee@mat.uc.pt}

\address[1]{CMUC, Department of Mathematics, University of Coimbra, Portugal}
\cortext[a]{Corresponding author}
\tnotetext[aaa]{All authors were  supported by Centro de Matem\'atica da Universidade de
   Coimbra -- UID/MAT/00324/2019, funded by the Portuguese Government
   through FCT/MEC and co-funded by the European Regional Development Fund
  through the Partnership Agreement PT2020. JG was partially supported by  FCT under Grant P2020 SAICTPAC/0011/2015. MS was supported by a PhD scholarship from FCT, grant PD/BD/128060/2016.}

\end{frontmatter}

%





\begin{abstract}
	 In 2005, Boman et al introduced the concept of factor width for a real symmetric positive semidefinite matrix. This is the smallest positive integer $k$ for which the matrix $A$ can be written as $A=VV^T$ with each column of $V$ containing at most $k$ non-zeros. The cones of matrices of bounded factor width give a hierarchy of inner approximations to the PSD cone. In the polynomial optimization context, a Gram matrix of a polynomial having factor width $k$ corresponds to the polynomial being a sum of squares of polynomials of support at most $k$. 	
	Recently, Ahmadi and Majumdar~\cite{Ahm14}, explored this connection for case $k=2$ and proposed to relax the reliance on sum of squares polynomials in semidefinite programming to sum of binomial squares polynomials (sobs; which they call  sdsos), for which semidefinite programming can be reduced to  second order programming  to gain  scalability at the cost of some tolerable loss of precision. With this they tap into the study of sobs that goes back to Reznick~\cite{reznick1987,reznick1989} and Hurwitz~\cite{Hurwitz}.
	In this paper, we will prove some results on the geometry of the cones of matrices with bounded factor widths and their duals, and use them to derive new results on the limitations of certificates of nonnegativity of quadratic forms by sums of $k$-nomial squares using standard multipliers. In particular we will show that they never help for symmetric quadratics, for any quadratic if $k=2$, and any quaternary quadratic if $k=3$. Furthermore we give some evidence that those are a complete list of such cases.
\end{abstract}

\section{Motivation and introduction}

Certifying nonnegativity of polynomials is a hard problem that has a long history in mathematical research. One of its most seminal results is Hilbert's famous result from 1888 \cite{hilbert1888} concerning certifying nonnegativity of polynomials by writing them as sums of squares of other polynomials. More precisely, he showed that every nonnegative degree $2d$ homogeneous polynomial in $n$ variables can be written as a sum of squares if and only if $d=1$, $n=2$ or $(d,n)=(2,3)$. Throughout the 20th century, sums of squares remained an active theoretical topic in real algebraic geometry, with many authors making significant contributions. Of special interest to us will be the result of Reznick from 1995 \cite{Reznick1995}, that states that any positive definite form of degree $2d$ and $n$ variables can be written as a sum of squares after being multiplied by a sufficiently high power of $G_n=(x_1^2+x_2^2+\dots+x_n^2)$. The rise of semidefinite programming at the turn of the millennium allowed the practical computation of sums of squares certificates \cite{Lasserre2001,parrilo2000}. This allowed their use for polynomial optimization purposes and gave an important applied dimension to the topic, giving rise to a number of important contributions to several different problems and dramatically increasing the interest in the area.

Despite all the success stories, the semidefinite programs that arise from sum of squares computations tend to grow exponentially with the degree of the polynomials involved, and even in low degrees do not scale well with dimension. This limits their direct application to large polynomial optimization problems and has given rise in recent years to several alternative proposals of cheaper to compute  certificates.

Ahmadi and Majumdar in their recent paper \cite{Ahm14} propose a new subclass of sums-of-squares polynomials to obviate these shortcomings. Instead of working with the full class of sum-of-squares polynomials they propose to work with polynomials they call diagonally dominant (dsos) or scaled diagonally dominant (sdsos) sums of squares, respectively, obtaining problems that are linear programs (LP) and second order cone programs (SOCP), respectively. We will be mostly interested in the class of sdsos polynomials that can be characterized as the class of polynomials that can be written as sums of squares of binomials, a class whose study goes back to Reznick~\cite{reznick1987,reznick1989} and Hurwitz~\cite{Hurwitz}. Note that since this is a more restrictive certificate than sums of squares, its power is more reduced, so in order to strengthen it and create a more powerful certificate, the authors also proposed to mimic Reznick's technique and consider a version of the certificates where we pre-multiply the target polynomial by some fixed power of $G_n$.

In this paper, we will focus on studying the gains obtained by the use of multipliers in this new family of certificates and their natural generalization: polynomials that can be written as sums of squares of $k$-nomials, a generalization already mentioned in \cite{Ahm14}. We will mostly restrict ourselves to the study of the quadratic case. We can summarize our findings in the following theorem.

\begin{thm}\label{prop:main_result}
	Let $p$ be a quadratic form on $n$-variables. If any of the following holds
	\begin{enumerate}
		\item $p$ is symmetric,
		\item $k=2$,
		\item $n=4$ and $k=3$,
	\end{enumerate}
then there is an $r$ such that $G_n^r p$ is a sum of squared $k$-nomials (so$k$s) if and only if $p$ itself is a so$k$s.
\end{thm}

Moreover we provide some evidence that the result fails for $n=5$ and $k=3$, hinting that the result might be a complete characterization of when do multipliers totally fail to help in certifying quadratics.

To prove such results one has to rely on matrix theory. Recall that a degree $2d$ form $p$ on $n$ variables is a sum of squares if and only if it can be written as $p(x)=z(x)_d^T H z(x)_d$ where $z(x)_d$ is the vector of all degree $d$ monomials and $H$ is a positive semidefinite matrix. This simple fact is what makes sums of squares certificates suitable for semidefinite programming. A similar result holds for sums of squared $k$-nomials: A degree $2d$ form $p$ on $n$ variables is so$k$s if and only if it can be written as $p(x)=z(x)_d^T H z(x)_d$ where $z(x)_d$ is the vector of all degree $d$ monomials and $H$ is a matrix with factor width at most $k$. While the geometry of the cone of positive semidefinite matrices is well understood, the cones of bounded factor width matrices are much worse understood, so in order to prove our results we have to start by studying their geometry.
\medskip

We organize the paper as follows:
In Section~\ref{2} we give some basic definitions and notations that will be used throughout the paper. In Section~\ref{3} we present the concept of factor width for positive semidefinite matrices and show how it connects to sums of squares of $k$-nomials. Then in Section~\ref{4} we give some geometric properties of the cone of bounded factor width matrices. In particular we characterize some of the extreme rays of their duals which will be used later  to derive the main results of the paper.

Section~\ref{5} draws inspiration from an example given by Ahmadi and Majumdar in~\cite{Ahm14}. They considered the polynomial $p_n^a=(\sum_{i=1}^n x_{i})^2+(a-1)\sum_{i=1}^n x_{i}^2$ and proved that for $n=3$, if $1 \leq a<2$, $G_n^r p_3^a$ is not a sum of squares of binomials for any $r$, although it is clearly nonnegative for $a\geq 1$. We generalize it, proving that for symmetric quadratic forms Reznick-type multipliers can't help, and thus proving part (1)  of Theorem \ref{prop:main_result}.

In Sections~\ref{6} and \ref{7}, we prove respectively parts (2) and (3) of Theorem \ref{prop:main_result}.  To complete the paper, in Section \ref{8} we give an example which numerically  suggests that our results are complete, as they cannot be extended in the most natural way to five or more variables. To that end, we give a  quadratic form in five variables which is not so$4$s but which becomes  so$4$s after multiplication with $\sum_{i=1}^5 x_{i}^2$.

\section{Definitions and notations} \label{2}

All our matrices are understood to be real.
We denote by $\cS^n,$ the $n\times n$ (real) symmetric matrices. A symmetric matrix $A$ is positive semidefinite (psd) if $x^T Ax\geq 0$ for all $x\in \R^n.$  This property will be denoted by the standard notation $A\succeq 0.$ By $\cS^n_+ $ we denote the subset of real symmetric  positive semidefinite matrices.  {A spectrahedron is a set that has an algebraic representation as the set of $x = (x_1,...,x_m)$ in $\R^m$ } such that $ A_0 + x_1 A_1 +\cdots+ x_m A_m\succeq 0$  where $A_0, . . . , A_m$ are real symmetric matrices. In particular, intersections of affine linear spaces with $\cS^n_+ $ are spectrahedra and every spectrahedron is linearly equivalent to such an intersection if we additionally impose that the $A_i$ must be linearly independent. It can easily be checked that intersections and products of spectrahedra are spectrahedra.

The Frobenius inner product for matrices $A,B\in \cS^n$ is given by $\langle A, B\rangle =\tr (AB{^T})=\sum_{i,j} A_{ij} B_{ij}.$   For a cone $K$ of matrices in $\cS^n$, we define its dual cone $K^*$ as $\{Y\in \cS^n :\langle Y, X\rangle \geq 0,\; \forall X\in K\}$.

For smooth reading the reader should  keep in mind the following basic facts found in texts about convex sets, for example in   \cite{Johnson}, or in \cite[Sections 1.3 and 1.4]{Ramana95}.
\begin{itemize}
\setlength\itemsep{0em}
\item[i.] If $C$ is a closed convex cone  then $C=C^{**}$ {where the dual cone to the cone $C$ is defined as $C^*=\{y: \langle x,y\rangle\geq 0\ \forall x \in C\}$ and $C^{**}=(C^*)^*$.}
\item[ii.]  $\langle A, S_1 BS_2\rangle= \langle S_1{^T} A S_2{^T}, B\rangle,$ whenever the matrix products are defined.
\item[iii.] The cone of real symmetric psd matrices is selfdual, i.e. $\cS_+^n=(\cS_+^n)^*.$
\item[iv.] If $A\in \cS^n_+$ and for some $x\in \R^n,$  $x{^T} Ax=0,$ then $Ax=0.$ See \cite[p. 463]{Johnson}.
\item[v.]  If $A\in \cS^n$ then $A$ is psd iff for  all psd matrices $B,$ $\langle A,B\rangle \geq 0.$ {In particular if $A,B\succeq 0,$ then  $\langle A,B\rangle \geq 0.$}
\item[vi.]  If $A,B\succeq 0,$  then $\langle A,B\rangle=0$ iff  $AB=0.$	
\end{itemize}

If $X=(x_{ij})$ is an $n\times n$ matrix and $K\subseteq \{1,2,...,n\},$ then  $X_K$ denotes the (principal) submatrix of $X$ composed from rows and columns of $X$ with indices in $K$ while $\text{supp}(X)$ denotes the support of $X$, i.e., the set $\{(i,j)\in \{1,2,...,n\}^2: x_{ij}\neq 0\}$ is the support of $X.$ If  $B$ is a $k\times k$ matrix and  $K\subseteq \{1,2,...,n\},$ a $k$ element subset of  $\{1,2,...,n\},$ then $\iota_K(B)$ means the $n\times n$ matrix $X$ which has zeros everywhere, except that $X_K=B.$ \medskip

We denote by $\R[x_{1:n}]=\R[x_1,...,x_n]$  the  algebra of polynomials in $n$ variables $x_1,x_2,\ldots,x_n$ over $\R.$  A {monomial} in $\R[x_{1:n}]$ is an expression
of the form $x^{\alpha}=x_1^{\alpha_1}x_2^{\alpha_2}\cdots x_n^{\alpha_n}$ and a polynomial $p$ in $\R[x_{1:n}]$ is a finite linear combination of monomials; so $p=\sum_{\alpha}c_{\alpha}x^{\alpha}$. A polynomial $p\in \R[x_{1:n}]$ is nonnegative if it takes only nonnegative values, i.e., $p(x)\geq 0,$ for all $x\in \R^n$ and a polynomial $p\in \R[x_{1:n}] $  is a {sum of squares (sos)} polynomial, if it has a representation $p=\sum_{i=1}^m q_i^2$
with  polynomials    $q_i\in \R[x_{1:n}]. $   Of course every sum-of-squares  polynomial is nonnegative and
every nonnegative polynomial has necessarily even degree, $2d,$ say. A useful introduction to polynomial optimization using sums of squares is found in \cite{SDOCAG}.

A {$k$-nomial}  is an expression of the form  $ \alpha_1 m_1+\cdots +\alpha_k m_k$  with $\alpha_1,...,\alpha_k$ reals and $m_1,...,m_k$  monomials.  Note that every $k-1$-nomial is also $k$-nomial. We call a sum of squares of $k$-nomials a {so$k$s}-expression.
A polynomial $p\in P_n$ is then called $r$-{so$k$s} if  $(\sum_{i=1}^nx_i^2)^r  p$ is so$k$s. Note that for $k=2$ these notions correspond to the scaled diagonally dominant sum of squares (SDSOS) and $r$-SDSOS notions introduced in \cite{Ahm14}.

\section{On the factor width of a matrix and sums of $k$-nomial squares}\label{3}	

The concept of  {factor width} of a real {psd} matrix $A$  was introduced by Boman et al.
in ~\cite{boman2005factor} as the   smallest integer $k$ such that  there exists a real
 (rectangular) matrix $V$ such that $A=VV{^T}$ and each column of $V$ contains at most $k$ non-zeros.
We let
\[FW^n_k=\{\text{{psd}}\; n\times n\; \text{matrices of factor width}\leq k\}.\]
We have of course
\[FW_1^n\subset FW_2^n\subset FW_3^n\subset \cdots \subset FW^n_n=\cS_+^n.\]
Next assume $A=VV^T$ is a {psd} matrix  where each column of $V$ has at most $k$
nonzero entries. By the rules of matrix multiplication, for any $i,j \in\{1,...,n\},$   and writing $V_{*\nu}$ and $V_{\nu *}$ for the $\nu$-th column or row of a matrix $V,$ respectively, we have
\[(VV^T)_{ij}=  \sum_{\nu=1}^m V_{i\nu}(V^T)_{\nu j}
=  \sum_{\nu=1}^m (V_{*\nu}V^T_{\hy \nu *})_{ij} =  \sum_{\nu=1}^m (V_{*\nu}V_{ * \nu}^{\hy T} )_{ij}.\]
Write  $A=\sum_{\nu=1}^m  (V_{*\nu}V_{ * \nu}^{\hy T} ).$ Note that each $V_{*\nu}V_{ * \nu}^{\hy T} $ is a {psd}
 $n\times n$ rank $1$ matrix  whose support  lies within a cartesian
product $K^2=K\times K$ for some $K\subseteq \{1,2,...,n\}$ of cardinality $k.$
Since   every $n\times n$ matrix with the latter properties can be written as $vv^T $ for some $v$
with at most $k$ nonzero entries, we have the following
\begin{prop}\label{factor}
	Let $A$ be an $n\times n$  {psd} matrix, and assume $k\in \mathbb{Z}_{\geq 1}.$   Then $A\in FW^n_k $ if and only if $A$ is the sum of  a finite family of {$n\times n$ symmetric  rank one psd matrices} whose  supports are all  contained in sets  $K\times K$ with $|K| = k.$
\end{prop}

{For $k=1$, $FW^n_1$ is simply the set of nonnegative diagonal $n \times n$ matrices.
The case $k=2$ corresponds to the cone of scaled diagonally dominant matrices studied in \cite{Ahm14}. That is the set of symmetric matrices $A$ such that there exists a diagonal matrix $D$ with positive diagonal entries such that $DAD$ is diagonally dominant.}

{As mentioned in the introduction, the study of the cones $FW^n_k$ is intimately connected to the study of sums of $k$-nomial squares. The reason can be seen in the next proposition. Denote by $z(x)_d$ the vector of all monomials of degree $d$, arranged in some order, in the variables figuring in $x.$
\begin{prop}\label{factork}
	A homogeneous multivariate polynomial $p(x)$ of degree $2d$ is a sum of $k$-nomial squares (so$k$s) if and only if it can be written in the form
	$ p(x)=z(x)_d^T Q z(x)_d$  with matrix $Q\in FW^{n+d-1 \choose d}_k.$
\end{prop}
\begin{proof}
	Consider  an expression  $a_1 m_1+\cdots + a_k m_k$	with reals $a_1,\ldots,a_k$ and monomials $m_1,\ldots,m_k.$  Note that monomials $m_1,\ldots,m_k$ occur
	necessarily in the column $z(x)_d$ at positions $i_1,\ldots,i_k,$ say.   Construct a column $q$ of size
	${n+d-1 \choose d}$  by putting into positions $i_1,\ldots,i_k$ respectively the reals $a_1,\ldots,a_k,$ and into all other positions 0s.  Then evidently $z(x)_d^T  q= a_1 m_1+\cdots + a_k m_k,$  and consequently
	$z(x)_d^T q q^T z(x)_d= (a_1 m_1+\cdots + a_k m_k)^2.$
	Consequently, a polynomial which is a sum of, say, $t$ squares of $k$-nomials can be written as $z(x)_d^T Q z(x)_d ,$
	where $Q=\sum_{\nu=1}^t q_\nu q_\nu^T ,$ with suitable columns $q_1,\ldots,q_t$ of size ${n+d-1 \choose d}$ each of which has at most $k$ nonzero entries. It follows that $Q$ is a matrix of factor width $k.$
	Conversely if $Q$ is of factor width $k,$ then we already know from the beginning of Section 3 that  we can write $Q=\sum_{\nu=1}^t q_\nu q_\nu^T$  where each column $q_\nu$ has at most $k$ nonzero real entries.  Clearly from the arguments above follows now that
	$z(x)_d^T Q z(x)_d$ yields a polynomial which is a finite sum of $k$-nomial squares.
\end{proof}}

Note that by setting $k=n$ we recover the standard fact that a homogeneous degree $2d$ polynomial $p(x)$ is a sum of squares if and only if $p(x)=z(x)_d^T Q z(x)_d$ for some psd matrix $Q$, so we should think of this as a natural refinement of that result.

 {From Proposition \ref{factor}, it follows  that} each set $FW_k^n$ is a convex closed subcone of $\cS^n_+.$ We will now focus on its  dual cone. From \cite[Lemma 5 and Subsection 3.2.5]{Permenter2017} we have the following result.

\begin{prop}\label{dualfactork}
	 The dual  of $FW_k^n$ is  given by
	\[ (FW_k^n)^* = \{X\in \cS^n \; |\; X_K\in \cS^k_+
	\text{ for all}\; K\subseteq \{1,2,...,n\} \text{ with } |K| =k \}.\]
	{where $X_K$ denotes the  principal submatrix of $X$ of support $K\times K$.}
	Furthermore the following inclusions and identity hold
		$$FW_k^n \subseteq \cS_+^n \subseteq (FW_k^n)^*   \text{ and  } FW_k^n=(FW_k^n)^{**}.$$
\end{prop}

In the next Section, we will focus on the study of the geometry of these dual cones, since that will play a fundamental role in proving that certain polynomials have no so$k$s representation.

\section{On the geometry of bounded factor width matrices}\label{4}
In this section, we give some geometric properties of the cone of bounded factor width matrices. In particular, we characterize some of the extreme rays of their duals. Extreme rays of the dual cones will play a fundamental role in our results as they will offer certificates of non-membership in the cones $FW_k^n$, which are crucial to prove that something is not so$k$s.

We will fully characterize in this section the extreme rays of $(FW_{n-1}^{n})^*$ for any $n$. Additionally we will characterize all the extreme rays of $(FW_{k}^{n})^*$ that are spanned by psd matrices and give a very concrete formulation for the extreme rays of $(FW_{3}^{4})^*$.

The first crucial observation is that all extreme rays of  $(FW_{k}^n)^*$ are exposed.
\begin{lem}\label{exposed}
	The cone $(FW_{k}^n)^*$ is (linearly equivalent to) a spectrahedron. Therefore any extreme face of $(FW_{k}^n)^*$ is exposed and in particular its rays.	
\end{lem}	
\begin{proof}
Let $I \subseteq \{1,...,n\}$ and $|I|=k$. Then $S(I)$, the set of symmetric $n\times n$ matrices such that the principal submatrix indexed by $I$ is positive semidefinite, is linearly equivalent to $\cS_{+}^k\times \R^{\binom{n+1}{2}-\binom{k+1}{2}}$ hence a spectrahedron. Since intersections of spectrahedra are spectrahedra,
$(FW_{k}^n)^*$ being the intersection of all such $S(I)$, implies that it is also a spectrahedron. The second part is a consequence of the theorem that every face of a spectrahedron is exposed. This is proved in  \cite[p.11]{Ramana95}.
\end{proof}

Since $S_+^n$ is a subset of $(FW_{k}^n)^*$ it makes sense to ask which rays of $S_+^n$ are extreme rays of $(FW_{k}^n)^*$. Interestingly it turns out that all extreme rays of $S_+^n$ are still extreme in $(FW_{k}^n)^*$.

\begin{lem}\label{extrmerayrank1}
	The matrix $A\in \cS_{+}^n$ spans an extreme ray of $(FW_{k}^n)^*$ if and only if it has rank $1$.
\end{lem}		
\begin{proof}
Note that if $A\in \cS_{+}^n$ does not span an extreme ray in $\cS_{+}^n$ then it can be written as a convex combination of elements in $\cS_{+}^n$ that are not multiples of $A$. Since these will still be elements of $(FW_{k}^n)^*$, $A$ will also not span an extreme ray in $(FW_{k}^n)^*$. So any psd matrix that spans an extreme ray of $(FW_{k}^n)^*$ must also span an extreme ray in $\cS_{+}^n$, which is equivalent to say it must have rank one.
	
	Now we prove that if the matrix $A$ has  rank one, then it spans an extreme ray of $(FW_{k}^n)^*$.
	Let $A=xx^T$ and assume now $A=X+Y$ with some $X,Y\in (FW_{k}^n)^*$ and some $x\in\R^n.$ Then for any $k$ element subset  $I\subseteq \{1,2,\ldots,n\}$, $x_Ix_I^T =X_I+Y_I.$
	By the characterization of $(FW_{k}^n)^*,$ $X_I,Y_I$ are {psd}; that is we have found in $\cS_{+}^k$ a representation of a rank $1$ matrix as a sum of two other matrices. Since the null space of a sum of two psd matrices is contained in the nullspace of each, we infer that $X_I$, $Y_I$ are multiples of $x_Ix_I^T$: for some real $\lambda_I$,
	$X_I=\lambda_I x_Ix_I^T,$   $Y_I=(1-\lambda_I) x_Ix_I^T.$
	Now, considering any two $k\times k$ submatrices of $X$ indexed by $I$ and $J$, we have if $i\in I\cap J$, then $x_{ii}=\lambda_I x_i^2=\lambda_J x_i^2$
	so if $x_{ii}\neq 0$ then $\lambda_I=\lambda_J$. Note that if $x_{ii}=0$, the entire $i$-th row and column of $X$ must be zero. For any $I$ and $J$ such that $i\in I$ and $j\in J$ with $x_{ii}\neq 0$ and $x_{jj}\neq 0$, we can pick a $k$-element set  $K$ such that $i,j\in K$ and the above argument gives $\lambda _I=\lambda_J=\lambda_K$. So all are equal to some $\lambda$ and  $X=\lambda x x^T.$
\end{proof}	

We now turn our attention to the particular case of $(FW_{n-1}^n)^*$. To characterize its extreme rays we will need the following auxiliary Lemma.

\begin{lem}\label{rankpsd}
	Assume that $A\in (FW_{n-1}^n)^*$. If for some size $(n-1)$ subset $I$ of $\{1,2,\ldots,n\}$, $\text{rank}(A_I)\leq n-3$ then $A$ is psd.
\end{lem}
\begin{proof}
	Since $A\in (FW_{n-1}^n)^*$, all its proper principal minors are nonnegative. So $A$ is psd if and only if $\det(A)\geq 0$. But by Cauchy's interlacing theorem, see \cite[p. 185]{Johnson}, if $\beta_1,\ldots, \beta _{n-1}$ are the (nonnegative) eigenvalues of $A_I$ and  $\gamma_1,\ldots, \gamma_n$ are the eigenvalues of $A$, then
	\[\gamma_1\leq \beta_1\leq  \gamma_2\leq \beta_2\leq\ldots\leq \beta_{n-1}\leq \gamma_n.\]
	Now, since $\text{rank}(A_I)\leq n-3$, $\beta_1$ and $\beta_2$ will be zero which leads to $\gamma_2=0$ and so $\det(A)=0$, hence $A$ is psd.	
\end{proof}	

We can now completely characterize the extreme rays of $(FW_{n-1}^n)^*$: they either are psd, a case that we characterized in Lemma \ref{extrmerayrank1}, or they have the following form.

\begin{thm}\label{extrmeraynotpsd}
	If the matrix $A\in (FW_{n-1}^n)^*$ is not psd, the matrix $A$ spans an extreme ray of $(FW_{n-1}^n)^*$ if and only if all of its  $(n-1)\times(n-1)$ principal submatrices have rank $n-2$.
\end{thm}

\begin{proof}
Suppose $A$ spans an extreme ray of $(FW_{n-1}^n)^*$. By Lemma \ref{rankpsd}, if any principal submatrix $A_I$ has rank smaller than $n-2$ then $A$ is psd, so it is enough to show that it must have rank at most $n-2$.
Suppose by contradiction that $A_{\{1,2,\ldots,n-1\}}$ is a principal submatrix of full rank. By Lemma~\ref{exposed}, the ray spanned by $A$ is an exposed extreme ray of $(FW_{n-1}^n)^*$, so there exists a $B\in (FW_{n-1}^n)^{**}=FW_{n-1}^n$ such that $\langle B,A\rangle=0$ and $\langle B,X\rangle>0$ for all $X\in (FW_{n-1}^n)^*\setminus\{\lambda A\;|\; \lambda \geq 0\}$.
	
	This $B\in FW_{n-1}^n$, and  so it can be written as
	\[B=\sum_{I\subseteq \{1,2,\ldots,n\}, |I|=n-1 }\iota_{I}(B(I)),\quad \text{for}\; B(I)\in \cS_+^{n-1}.\]
	
	We thus get
	\[0=\langle B,A\rangle=\sum_{I\subseteq \{1,2,\ldots,n\}, |I|=n-1 }\langle\iota_{I}(B(I)),A\rangle=\sum_{I\subseteq \{1,2,\ldots,n\}, |I|=n-1 }\langle B(I),A_I\rangle.\]
	Since the  $(n-1)\times (n-1)$ principal submatrices of $A$ are all {psd}, we get that all the inner products are nonnegative and hence must be $0$, which means $\langle B(I),A_I\rangle=0$ for all $I.$ 	
	Under the current supposition that $A_{\{1,2,\ldots,n-1\}}$ is not singular, we conclude that $B({\{1,2,\ldots,n-1\}})=0$.
	
	Let now $a$ be the $n$-th column of $A$ and let  $ \tilde A= aa^T.$ Of course $\tilde A\in \cS_+^n$ and so $\tilde A\in (FW_{n-1}^n)^*.$
	We have
	\[\langle \iota_{I}(B(I)), \tilde A\rangle = \langle \iota_{I}(B(I)),aa^T \rangle = \langle B(I), a_I a_I^T\rangle.\]
	But note that $a_I$ is a column of $A_I$ for $I\neq\{1,2,\ldots,n-1\}$, so the row space of $\tilde A_I=a_I a_I^T$ is contained in that of $A_I$ hence there exists some $\lambda > 0$ such that
$A_I=\lambda a_Ia_I^{T}+A_I^\prime$ for some $A_I^\prime\succeq 0$ (see Theorem 8.6.2 in \cite{bernstein2009matrix}) and $\langle B(I), A_I\rangle=0$ implies $\langle B(I), a_I a_I^T\rangle=0$.
	Since we know already $B({\{1,2,\ldots,n-1\}})=0$ we get $\langle B,\tilde A \rangle=0$ . Now evidently $\tilde A$ is not a multiple of $A$ so it does not span the same  ray and we have a contradiction to our assumption that $A_{\{1,2,...,n-1\} }$ has full rank. Therefore  $A_{\{1,2,...,n-1\} }$ has rank at most $n-2$ and
	similarly any other principal $(n-1)\times (n-1)$ submatrix has rank exactly $n-2.$
	
	For the reverse direction, assume that $A$ does not span an extreme ray of 	$(FW_{n-1}^n)^*.$ This means that we can write it as
	\[A=\gamma X+(1-\gamma)Y\ \text{for some distinct }\ X,Y \in (\widetilde{FW}_{n-1}^n)^*\ \text{and}\; \gamma\in ]0,1[,\]  where $(\widetilde{FW}_{n-1}^n)^*$ is the compact section of the cone $(FW_{n-1}^n)^*$ consisting of the matrices that  have the same trace as matrix $A$.
	
	Let $X_{\lambda}=\lambda X+(1-\lambda)Y,\; \lambda\in \R$. Given some $I$, we know that $(X_\lambda)_I$ has rank at most $n-2$: in fact, there is a $1$ dimensional space, $\ker(A_I)$, which is always contained in $\ker(X_\lambda)_I$. Then the set $L=\{\lambda|X_{\lambda}\in (\widetilde{FW}_{n-1}^n)^*\}=[\lambda_{min},\lambda_{max}]$ since $(\widetilde{FW}_{n-1}^n)^*$ is compact. The eigenvalues and eigenvectors of $(X_\lambda)_I$ change continuously with $\lambda$ and since one eigenvalue corresponds to a fixed eigenvector, the only way for $(X_\lambda)_I$ to stop being psd is if a second eigenvalue switches from positive to negative, so the rank of $(X_\lambda)_I$  is at most $n-3$. Since $X_{\lambda_{max}}$ is in the boundary, there must be some $I$ for which this happens, hence by  Lemma~\ref{rankpsd} is psd, and similarly  for $X_{\lambda_{min}}$ with a possibly different $I$. Hence $A$ is psd since it is a convex combination of  $X_{\lambda_{max}}$ and  $X_{\lambda_{min}}$.	This is a contradiction to the hypothesis.
\end{proof}

Recall that the operation $X  \mapsto Q^T X Q$ is an automorphism of $\cS^n$ if $Q$ is invertible.
It is also clear that it restricts to an automorphism of $\cS_+^n$, but not necessarily to one of $(FW_k^n)^*$. In the next observation we give some simple sufficient conditions for that to be true. This will allow us to write the extreme rays in some canonical form.

\vs\noindent
{\bf Observation.}  
The operation $X  \mapsto Q^T X Q$ restricts to an automorphism of $(FW_k^n)^*$ if $Q$ is a positive definite diagonal matrix or a permutation matrix. In particular, it preserves extreme rays.
\begin{proof}
Since the operation is invertible, we just need to show it preserves $(FW_k^n)^*$.
$A\in (FW_k^n)^*,$ if it can be written as $VV^T$ with at most $k$-nonzero entries in each column of $V$. But the image of $A$ is then $Q^TV (Q^TV)^T$ and the columns of $Q^TV$ are just $Q^T$ times the columns of $V$. Since multiplying a vector by a permutation or a positive definite diagonal matrix preserves the number of nonzero entries, we have that the image of $A$ is still in $(FW_k^n)^*$.
\end{proof}	

Based on the results that we have proven so far, we can give in explicit form  the complete list of extreme rays of $(FW_3^4)^*$ .
\begin{prop}\label{matrixform}
	Let $B$ be a symmetric $4\times 4$ non {psd}  matrix which spans an extreme ray of $(FW_3^4)^*$, then for some $a,c \in ]-\pi,\pi[\setminus\{0\}$ some permutation $P$ and some nonsingular diagonal matrix $D$,
	\[DPBP^{T}D= \begin{bmatrix}
	1   &\cos(a) &\cos(a-c)   &\cos(c)\\
	\cos(a)& 1  &\cos(c) & \cos(a- c)\\
	\cos(a-c)& \cos(c) &1        &\cos(a) \\
	\cos(c)& \cos(a- c) &\cos(a) &1
	\end{bmatrix}.\]
\end{prop}
\begin{proof}
	First note that by the above observation, there is a scaling that takes all diagonal entries of $B$ to $1$. Furthermore, by assumption, $B\in (FW_3^4)^*$ which means all of its $3\times 3$ and accordingly its $2\times 2$ principal submatrices are psd, hence for all $i,j \in \{1,2,3,4\}$, $0\leq b_{ii}b_{jj}-b_{ij}^2 =1 -b_{ij}^2$ and hence $b_{ij}^2\leq 1$ for all pairs $(i,j)$. Therefore, using that the image of the cosine function is $[-1,1],$ we can write
	\[DBD=\begin{bmatrix}
	1     & \cos(a)& \cos(b) & \cos(c) \\
	\cos(a)& 1     & b_{23}& b_{24}\\
	\cos(b)& b_{23}& 1     & b_{34} \\
	\cos(c)& b_{24}& b_{34}& 1
	\end{bmatrix},\]
	for some $a,b,c \in [-\pi,\pi]$ and some diagonal matrix $D$. From now on we assume that $B$ has this form. The possibilities, $a,b,c\in \{-\pi,0,\pi\}$ will be excluded below. Now since $B$ spans an extreme ray of $(FW_3^4)^*$, by Theorem~\ref{extrmeraynotpsd} all of its $3\times 3$ principal submatrices have rank $2$ and hence have zero determinant. Hence by starting with principal submatrix $B_{123}$, we have 		
	\[0=\text{det}\left(\begin{bmatrix}
	1&\cos(a) &\cos(b)\\
	\cos(a)& 1& b_{23} \\
	\cos(b)&b_{23}&1
	\end{bmatrix}\right)= 1 - b_{23}^2 - \cos(a)^2 + 2b_{23}\cos(a)\cos(b) - \cos(b)^2.\]
	By solving this quadratic equation for $b_{23}$ one finds
	
	$$\begin{array}{rcl}	
	b_{23}&\in &\{\cos(a)\cos(b) \pm \sqrt{1-\cos(a)^2-\cos(b)^2+\cos(a)^2\cos(b)^2} \}\\
	&=& \{\cos(a)\cos(b) \pm \sqrt{(1-\cos(a)^2) (1-\cos(b)^2)} \}\\
	&=& \{\cos(a)\cos(b) \pm \sin(a)\sin(b) \}\\
	&=&\{\cos(a\mp b)\}. 	
	\end{array}$$
	
	We do completely analogous calculations for principal submatrices  $B_{134}$  and  $B_{124}$ and obtain
	$b_{34}\in \{\cos(b\pm c)\}$ and
	$b_{24}\in \{\cos(a\pm c)\}$,  respectively.
	Now we have eight matrices that emerge from choosing one of the  symbols $+$ or $-$  in each of the patterns  $a\pm b, a\pm c, b\pm c$ existent in the matrix below by taking care that the symmetry of the matrix is preserved.
	\[\begin{bmatrix}
	1   &\cos(a) &\cos(b)   &\cos(c)\\
	\cos(a)& 1  &\cos(a\pm b) & \cos(a\pm c)\\
	\cos(b)& \cos(a\pm b) &1        &\cos(b\pm c) \\
	\cos(c)& \cos(a\pm c) &\cos(b\pm c) &1
	\end{bmatrix}.\]
	The following table indicates in the first column the possible selections of signs in $ a\pm b, a\pm c, b \pm c,$ respectively;
	and  in the second column and the third column the determinants of the respective matrices  $ B_{234}$  and $ B.$
	\[
	\begin{array}{ccc}
	x\pm y&  \det( B_{234})                                 &     \det(B) \\
	+, +, +& 4\sin(a) \sin(b) \sin(c) \sin(a + b + c)&   -4 \sin(a)^2 \sin(b)^2 \sin(c)^2\\
	+, +, -& 0 & 0\\
	+, -, +& 0 & 0\\
	+, -, -& -4\sin(a) \sin(b) \sin(a + b - c) \sin(c)&  -4\sin(a)^2\sin(b)^2\sin(c)^2\\
	-, +, +& 0& 0\\
	-, +, -& -4\sin(a) \sin(b) \sin(c) \sin(a - b + c)& -4\sin(a)^2 \sin(b)^2 \sin(c)^2\\
	-, -, +& 4 \sin(a) \sin(b) \sin(a - b - c) \sin(c)  &   -4\sin(a)^2 \sin(b)^2 \sin(c)^2\\
	-, -, -& 0&  0
	\end{array}
	\]
	Now assume one of the reals $a,b,c$ is $0$ or $\pi.$
	Then the table shows that all entries in columns two and three vanish. Hence the matrix $B$ in this case is {psd}. Thus in order that $B$, as required, is not {psd} it is necessary that $a,b,c\not \in\{-\pi,0,\pi\}$.  In this case column 3 guarantees we do not get a  {psd} matrix $B$ in exactly the cases of the sign choices  $+++, +--,-+-, --+$  for $ a\pm b, a\pm c, b \pm c,$ respectively. The matrices corresponding to rows, 2,3,5,8 of the table are {psd} independent of	choices $a,b,c$.
	Explicitly this means that $B$ must be one of the following four matrices

	\[\hspace*{-1cm}\begin{bmatrix}
	1   &\cos(a) &\cos(b)   &\cos(c)\\
	\cos(a)& 1  &\cos(a+ b) & \cos(a+ c)\\
	\cos(b)& \cos(a+ b) &1        &\cos(b+ c) \\
	\cos(c)& \cos(a+ c) &\cos(b + c) &1
	\end{bmatrix},
	\begin{bmatrix}
	1   &\cos(a) &\cos(b)   &\cos(c)\\
	\cos(a)& 1  &\cos(a+ b) & \cos(a- c)\\
	\cos(b)& \cos(a+ b) &1        &\cos(b- c) \\
	\cos(c)& \cos(a- c) &\cos(b- c) &1
	\end{bmatrix}, \]
	\[\hspace*{-1cm} \begin{bmatrix}
	1   &\cos(a) &\cos(b)   &\cos(c)\\
	\cos(a)& 1  &\cos(a- b) & \cos(a+ c)\\
	\cos(b)& \cos(a- b) &1        &\cos(b- c) \\
	\cos(c)& \cos(a+ c) &\cos(b- c) &1
	\end{bmatrix}, \begin{bmatrix}
	1   &\cos(a) &\cos(b)   &\cos(c)\\
	\cos(a)& 1  &\cos(a- b) & \cos(a- c)\\
	\cos(b)& \cos(a- b) &1        &\cos(b+ c) \\
	\cos(c)& \cos(a- c) &\cos(b+ c) &1
	\end{bmatrix}.\]
	
	Note by substituting the letter $c$ by $-c$ in the left upper matrix we get the right upper matrix because $\cos(-c)=\cos( c).$ Exactly the same remark leads from the left lower matrix to the right lower matrix. Finally note that after doing the transpositions of rows and columns $3,4$, the upper left matrix shown takes the form
	\[\begin{bmatrix}
	1   &\cos(a) &\cos(c)   &\cos(b)\\
	\cos(a)& 1  &\cos(a+ c) & \cos(a+ b)\\
	\cos(c)& \cos(a+ c) &1        &\cos(b+ c) \\
	\cos(b)& \cos(a+ b) &\cos(b + c) &1
	\end{bmatrix}\]
	and after changing the name of variable $c$ to $-b$ and of variable $b$ to $c$   and  noting  that $\cos(b-c)=\cos(c-b)$ we see we have obtained the left lower matrix. Hence we have one form and its possible permutations. We focus at the right lower matrix as the standard. Thus we have showed that after applying a suitable permutation $P$ we have
	\[P^TBP=	\begin{bmatrix}
1   &\cos(a) &\cos(b)   &\cos(c)\\
	\cos(a)& 1  &\cos(a- b) & \cos(a- c)\\
	\cos(b)& \cos(a- b) &1        &\cos(b+ c) \\
	\cos(c)& \cos(a- c) &\cos(b+ c) &1
	\end{bmatrix}.\]			
	Once again we will now assume that $B$ has this precise form.
		 Now we know that the determinant of the submatrix  $B_{234}$ is $ 4\sin(a)\sin(b)\sin(a - b - c)\sin(c). $ We know by Theorem~\ref{extrmeraynotpsd} that all $3\times 3$ principal minors must vanish, so $\det(B_{234})=0$ which happens if and only if  $b=a-c+k\pi$.  Substituting this in the start matrix $B$ we get the following two forms
	\[ \begin{bmatrix}
	1   &\cos(a) &\delta\cos(a-c)   &\cos(c)\\
	\cos(a)& 1  &\delta\cos(c) & \cos(a- c)\\
	\delta\cos(a-c)& \delta\cos(c) &1        &\delta\cos(a) \\
	\cos(c)& \cos(a- c) &\delta\cos(a) &1
	\end{bmatrix},\]
	with $\delta =\pm 1$. But note that these are the same up to scaling by the diagonal matrix $\Diag(1,1,-1,1).$ So we may assume $\delta=1$, finishing the proof.
\end{proof}	

\section{Symmetric quadratics and sums of $k$-nomial squares}\label{5}

Recall that writing a polynomial as a sum of squares of $k$-nomials provides a certificate for non-negativity. Since it is a weakening of the full strength sum of squares, it is clear that it is only a sufficient condition. In the case of sums of squares, however,  Reznick \cite{Reznick1995} showed that any positive definite form in $n$ variables has a sum of squares certificate if multiplied by a sufficiently high power of $G_n=(x_1^2+\cdots+x_n^2)$. This shows that for positive polynomials we can always get sums of squares certificates of nonnegativity, provided we are willing to bump up the degree. This motivates the analogous construction for so$k$s: we say
that a polynomial $p$ in $n$ variables is $r$-so$k$s if $G_n^rp$ is a so$k$s. This can in fact help, for instance one can show that the famous Motzkin polynomial is $2$-so$2$s, but it turns out that it does not capture all positive polynomials.

Ahmadi and Majumdar in~\cite{Ahm14} considered the symmetric quadratic forms
\[p_n^a=(\sum_{i=1}^n x_{i})^2+(a-1)\sum_{i=1}^n x_{i}^2\]
when $n=3$ and proved that if $a<2$ then no nonnegative integer $r$ can be chosen so that
$(x_1^2+x_2^2+x_3^2)^r p_3^a$   is a sum of squares of binomials, although {$p_n^a$} is clearly nonnegative for $a\geq 1$.
In this section, we extend their negative result along the same lines, to include all symmetric quadratic forms and all $k$'s. We will show that $p_n^a$ is $r$-so$k$s for some $r$ if and only if it is so$k$s.

In order to do so, we start by a lemma, that will be used throughout our results, that relates the coefficients of a quadratic form $q$ to those of $G_n^rq$.

\begin{lem}\label{coefficients}
	Consider a quadratic form $q(x)=x^{T}Qx$  and a polynomial $p$  related to $q$ by
	$p=(\sum_{i=1}^n (\lambda_i x_i)^2)^r\,q.$   Then
	every monomial of $p$ has at most two odd degree variables and we have $p_{(i,j)}=2 (\sum_{i=1}^n \lambda_i^2)^r q_{ij}$ and
	$p_0=(\sum_{i=1}^n \lambda_i^2)^r \tr(Q)$ where $p_{(i,j)}$ is the sum of coefficients  of the monomials in which $x_{i}$ and $x_{j}$	have odd degree, $p_0$ is the sum of coefficients of monomials in which every variable has even degree and $q_{ij}$ is the entry $(i,j)$ of $Q$.
\end{lem}
\begin{proof}
	The quadratic form is 	
	\[ q(x)=\sum_{1\leq i,j\leq n}x_{i}q_{ij}x_{j}=\sum_{i=1}^{n}q_{ii}x_{i}^{2}+\sum_{1\leq i<j\leq n} 2q_{ij}x_{i}x_{j},\]
	while by the  multinomial theorem  we have
	\[( (\lambda_1 x_{1})^2+\cdots + (\lambda_n x_{n})^2)^r=
	\sum_{i_{1}+\cdots+i_{n}=r}  \binom{r}{i_{1},\ldots,i_{n}}
	(\lambda_1 x_{1})^{2i_{1}}(\lambda_2 x_{2})^{2i_{2}}\ldots (\lambda_n x_{n})^{2i_{n}}.\]
	Thus, {by looking to the $\lambda$s after multiplication}, from the definition of $p,$ we get
	\begin{align*}
	p=&\sum_{(i,\underline{i})\in J_{1}} q_{ii}
	\binom{r}{\underline{i}}  \lambda_{1}^{2i_{1}}\cdots  \lambda_{n}^{2i_{n}} \cdot
	x_{1}^{2i_{1}}\cdots x_{i}^{2i_{i}+2}\cdots x_{n}^{2i_{n}}\\
	&+\sum_{((i,j),\underline{i})\in J_{2}} 2q_{ij}  \binom{r}{\underline{i}}
	\lambda_{1}^{2i_{1}}\cdots \lambda_{i}^{2i_{n}}\cdot
	x_{1}^{2i_{1}}\ldots x_{i}^{2i_{i}+1}\ldots x_{j}^{2i_{j}+1}\ldots x_{n}^{2i_{n}},
	\end{align*}
	where $\underline{i}=(i_{1},\ldots,i_{n})$ and, with $|\underline{i}|=i_{1}+\cdots+i_{n}$,
	\begin{align*}
	J_{1}&=\{(i,\underline{i}): i\in \{1,\ldots,n\},\underline{i}\in \mathbb{Z}_{\geq 0}^n,|\underline{i}|=r\},\\
	J_{2}&=\{((i,j),\underline{i}): 1\leq i<j\leq n,\underline{i}\in \mathbb{Z}_{\geq 0}^n,|\underline{i}|=r\}.
	\end{align*}
	From the above equation for $p$ , we recognize that
	\[ p_{(i,j)}=2q_{ij}\sum_{i_{1}+\cdots+i_{n}=r} \binom{r}{i_{1},\ldots,i_{n}}
	\lambda_{1}^{2i_{1}}\cdots \lambda_{i}^{2i_{n}}
	=2 q_{ij} (\lambda_1^2+\cdots +\lambda_n^2)^r ,\]
	again by  the multinomial theorem; and  similarly we have
	\[\ny\ny\ny p_0=
	\sum_{i=1}^{n}q_{ii}\sum_{i_{1}+\cdots+i_{n}=r}  \binom{r}{i_{1},\ldots,i_{n}}
	\lambda_{i}^{2i_{1}} \cdots \lambda_{n}^{2i_{n}}=(\lambda_1^2+\cdots+\lambda_n^2 )^r \tr (Q).\]
\end{proof}

We are interested in characterizing when the quadratic forms $p_n^a$ are sums of $k$-nomial squares. Note that $p_n^a$, can be written as $a \sum_{i=1}^n x_i^2 +2\sum_{i<j}x_ix_j$, so $p_n^a= z(x)^T  Q z(x)$ where $Q$ is the $n \times n$ matrix with $a$'s in the diagonal and $1$'s in the off-diagonal entries. So $Q$ can be seen as a diagonal matrix perturbed by a rank one matrix, and in those cases we can easily compute the determinant (see for instance Muir's treatise~\cite[p 59]{muir1933treatise}).
\begin{lem}\label{determinant}
Consider the $m \times m$ matrix that is $b$ in the diagonal and $c$ in the off-diagonal entries.
Then its determinant is $(b-c+cm)(b-c)^{m-1}$.
\end{lem}
Since all the principal submatrices of $Q$ have this precise form, we can use this result to characterize when $p_n^a$ is so$k$s.

\begin{prop}
	If   $a\geq \frac{n-1}{k-1},$ then $p_n^a$ is a sum of $k$-nomial squares.
\end{prop}
\begin{proof}
	Let $Q$ be the $n \times n$ matrix with $a$'s in the diagonal and $1$'s in the off-diagonal entries as above.
	There exist  ${n\choose k}$ subsets $K$ of cardinality $k$ of the set $\{1,2,\ldots,n\}.$  Let $i,j\in \{1,2,\ldots,n\}.$
	A pair $(i,i)$ lies in exactly ${n-1 \choose k-1}$
	of the sets $K\times K$ while a pair $(i,j)$ with $i\neq j$ lies in $K\times K$ if and only if $\{i,j\}\subseteq K.$ It hence lies in exactly
	${n-2 \choose k-2}$ sets $K\times K.$  Consider the $k\times k$ matrix
	\[
	B={n-2\choose k-2}^{-1}
	\begin{bmatrix}
	\frac{(k-1) a}{n-1} & 1 &  \cdots  &1 &  1    \\
	1 &  \frac{(k-1) a}{n-1} & \cdots   &1  & 1     \\
	\vdots& & \ddots &     &\vdots \\
	1 &      &          &      &  1  \\
	1  & 1 & \cdots & 1 & \frac{(k-1) a}{n-1}
	\end{bmatrix},\]
	and recall that $\iota_K(B)$ is the $n\times n$ matrix of support   $K\times K$ which carries on it the  matrix $B.$   Then our arguments yield that $\sum_{K: |K|=k} \iota_K(B) =Q.$
	
	Take an arbitrary $l\times l$ principal submatrix of the matrix factor of $B.$  By the previous lemma, this submatrix has determinant $( \frac{(k-1)a}{n-1}-1+l)( \frac{(k-1)a}{n-1} -1)^{l-1} .$ It follows from the hypothesis for $a$ that
	this  determinant is nonnegative.   So $B,$ and thus $\iota_K(B),$ is a {psd} matrix and $Q$ hence a matrix of factor width $\leq k$   by Proposition~\ref{factor}. This means by Proposition~\ref{factork} that $p_n^a$ is a sum of $k$-nomial squares.
\end{proof}

Note that using a simple symmetrization argument one can prove that the condition is not only sufficient but also necessary, but we will not need this fact in this paper. We are now ready to state and prove the main result of this section, that says that for any symmetric quadratic form, using Reznick-type multipliers does not increase the power of so$k$s.

\begin{thm}\label{thm:generalAhmadiMajumdar}
	For integers $n\geq 0$ and $r\geq 0,$ define
	
	\[p_{n,r}^a=(\sum_{i=1}^n x_{i}^2)^r\,p_{n}^a
	=(\sum_{i=1}^n x_{i}^2)^r  \cdot \left( (\sum_{i=1}^n x_{i})^2+(a-1)\sum_{i=1}^n x_{i}^2\right).\]
	Then $p_{n,r}^a$ is a sum of $k$-nomial squares if and only if  $p_n^a=p_{n,0}^a$ is a sum of $k$-nomial squares.
\end{thm}
\begin{proof}
	Clearly, if $p_n^a$ is a so$k$s then $p_{n,r}^a$ is a so$k$s. So we need  to show the {converse}.
	Assume that the degree $2(r+1)$ polynomial $p_{n,r}^a$ is a so$k$s. Let
	${E}_{n,r+1}=\{(i_1,\ldots,i_n)\ s.t.\ i_k\in \mathbb{Z}_{\geq 0}^n, \sum_{k=1}^n i_k=r+1\}$
	 be the set of vectors of exponents in $\mathbb{Z}_{\geq 0}^n$ that occurs in the family of  monomials of a homogeneous
	 polynomial of degree $r+1$ in variables
	$x_1,...,x_n.$ Let this family of monomials be also the one that occurs in $z(x)_{r+1}$.
	
	By Proposition~\ref{factork}, we can write
	\[p_{n,r}^a=z(x)_{r+1}^T H_{n,r} z(x)_{r+1}\;\text{for some}\; H_{n,r} \in FW_{k}^{ {n+r \choose r +1} }.\]
	Call an $i\in \mathbb{Z}_{\geq 0}^n$  {even} if it has only even entries and
	consider now the matrix ${B_{n,r}\in \R^{{E_{n,r+1}}\times{E_{n,r+1}}}}$ given by
	
	\[(B_{n,r})_{ij} =  \left\{ \begin{array}{cl}
	k-1  & \text{ if $i+j$ is even, }    \\
	-1  & \mbox{ otherwise. }
	\end{array}  \right . \]
	We will show now that  $B_{n,r}\in (FW_{k}^{{n+r \choose r+1}})^{*};$ that is we shall prove
	that every $k\times k$ principal submatrix of $B_{n,r}$  is {psd}.
	Since $n,r$ are fixed, we write $B$ and $H$ for matrices $B_{n,r}, H_{n,r}$ respectively.
	
	Note that a sum $i+j$ of such $n$-{tuples} is even if and only if the sets of positions in
	$i$  where odd entries occur equals  the corresponding set in $j.$
	(Example: The 5-{tuple} $i=(1,0,0,3,2)$ has $\{1,4\}$ as the set of positions of odd entries.)
	
	So take a $k\times k$ submatrix $M$ of $B$ with rows and columns indexed by the $n$-{tuples}
	$i_{1},\ldots,i_{k},$ say. Determine  for each $n$-{tuple} its  set of positions of odd entries. Let
	$S_1,\ldots,S_l$ $(l\leq k)$ be the distinct non empty sets of such positions. Now rearrange
	the $n$-{tuples} so that the first few $n$-{tuples} each have $S_1$ as  set of positions of odd entries,
	the next few have $S_2$ as such set of positions, etc.   Let $s_1,\ldots,s_l$ be the
	sizes of these sets.
	To the rearrangement of the $n$-{tuples} corresponds a $k\times k$ permutation
	matrix $P$ such that $PMP^T$ is a block matrix with constant blocks $k-1$ in the diagonal and $-1$ in the off-diagonal blocks. One can then write $PMP^T= C^T N C$ where $N$ is a $l \times l$ matrix with  $k-1$ in the diagonal and $-1$ in the off-diagonal entries and $C$ is an $l \times k$
matrix where every column has a single non-zero entry: the first $s_1$ columns have $1$ in the first row, the next $s_2$ columns have $1$ in the second row, and so on, until the last $s_l$ columns that have $1$ in the $l$-th row.
Now, again by Lemma~\ref{determinant}, $N$ is  {psd},
	Hence $M$ will be psd. Since the $k\times k$ submatrix $M$ of $B$ was arbitrary, we are done with proving
	that   $B\in (FW_k^{n+r\choose r+1})^*.$
	
	By definition of the concept of  a dual cone,
	we have $\langle B, H\rangle \geq 0,$ hence
		\[\langle B, H\rangle = (k-1)\sum_{i,j: i+j\;  \text{even}} h_{ij}+ (-1) \sum_{i,j: i+j\; {\text{has an odd entry}}} h_{ij} \geq 0.\]
	Since  the quadratic form underlying our construction of $p_{n,r}^a$ is
	$p_n^a=a \sum_{i=1}^n x_i^2 +2\sum_{i<j}x_ix_j,$  and it has the defining matrix $Q$ mentioned in the previous proposition, we get  by Lemma~\ref{coefficients} that
\begin{eqnarray*}	
\sum_{i,j: i+j\; \rm even}h_{ij}&=&n^r \text{trace} \; Q = n^{r+1} a; \\ 
\sum_{i,j: {{i+j \text{ has an}}\atop{\text{ odd entry}}}} h_{ij} &= & 2 n^r \times \sum_{1\leq i<j \leq n}  q_{ij} = 2n^r  \frac{1}{2}n(n-1)=n^{r+1} (n-1).
\end{eqnarray*}
	Hence the inequality above reads  $(k-1) n^{r+1} a \geq n^{r+1}(n-1)$ or  $a\geq \frac{n-1}{k-1},$  which means by the previous proposition that $p_n^a$ is a sum of $k$-nomial squares.
\end{proof}

\section{General quadratics and sums of binomial squares}\label{6}

For the case of $k=2$, sums of squares of $k$-nomials are also known as sums of binomial squares~\cite{Fidalgo2011} (sobs) or scaled diagonally dominant sums of squares (SDSOS)~\cite{Ahm14}. In this section we will try to generalize Ahmadi {and} Majumdar's counterexample in this setting. More concretely we will prove that the standard multipliers are useless for certifying nonnegativity of quadratics using sobs, as we prove that  a quadratic form is $r$-sobs, if and only if it is sobs. But before we proceed further, we shall point that there is a simple characterization of the existence of sobs certificates for quadratics.

\begin{prop}\label{dmtsobsprop}
	Given a quadratic form
	$q(x)=\sum_{i=1}^{n}q_ix_i^{2}+\sum_{i<j}q_{ij}x_i x_j$, then if $\hat{q}(x)=\sum_{i=1}^{n}q_ix_i^{2}-\sum_{i<j}|q_{ij}|x_i x_j$ is nonnegative, $q(x)$ is a sum of binomial squares.
\end{prop} 	

Note that this follows immediately from classic elementary properties of symmetric $M$-matrices. In this form it can be found explicitly for instance in  \cite[Corollary 2.8]{Fidalgo2011}.

This is enough to show the previously announced result.

\begin{thm}\label{Quadraticsobs}
	Let $q(x)=q(x_1,\ldots,x_n)$ be a real quadratic form and let $r\in \mathbb{Z}_{\geq 0}$. Then {$q(x)(x_1^2+\cdots + x_n^2)^r$ is a sum of  binomial squares, if and only if $q(x)$ is itself a sum of  binomial squares.}
	
\end{thm}
\begin{proof}
	{If $q(x)$ is a sum of  binomial squares, then clearly the polynomial $q(x)(x_1^2+\cdots + x_n^2)^r$ is a sum of  binomial squares. For the converse,} assume that $q(x)(x_{1}^2+\cdots+ x_{n}^2)^r$ is a sum of binomial squares.   Write $q(x)=\sum_{i=1}^{n}a_ix_i^2+\sum_{1\leq i<j\leq n} d_{ij}x_ix_j,$ say.
	Then considerations as in the proof of Lemma~\ref{coefficients} yield
	\begin{align*}
	q(x).(x_{1}^2+\cdots + x_{n}^2)^r
	= &\sum_{(i,\underline{i})\in J_1}
	a_i \binom{r}{\underline{i}} x_{1}^{2i_{1}} \cdots x_{i}^{2i_{i}+2} \cdots x_n^{2i_n}\\
	&+\sum_{((i,j),\underline{i})\in J_2} d_{ij} \binom{r}{\underline{i}}  x_1^{2i_1}\cdots x_i^{2i_i+1}\cdots x_j^{2i_j+1}\cdots x_n^{2i_n},
	\end{align*}
	where again,
	\[J_{1}=\{(i,\underline{i}): i\in \{1,\ldots,n\},\underline{i}\in \mathbb{Z}_{\geq 0}^n,|\underline{i}|=r\},\]
	\[J_{2}=\{((i,j),\underline{i}): 1\leq i<j\leq n,\underline{i}\in \mathbb{Z}_{\geq 0}^n,|\underline{i}|=r\}.\]
	Now the monomials of degree $r$ are of the form $x_1^{i_1}x_2^{i_2}\ldots x_n^{i_n}$ with $i_1+\cdots+i_n=r$. There are as we know $L=\binom{r+n-1}{r}$ such monomials. We order these and denote them by $m_1,\ldots,m_L$. Every binomial is of the form $(\alpha_{ij}m_i+\beta_{ij}m_j)$ with some selection of $i,j$ with $1\leq i<j\leq L$. By defining suitable $\alpha_{ii}$, we can thus assume the binomials are of the form $\alpha_{ii}m_i$,
	$1\leq i\leq L$ or $(\alpha_{ij}m_i+\beta_{ij}m_j)$ with $1\leq i<j\leq L$. A sum of binomial squares is thus given as 	

\begin{eqnarray*}
\lefteqn{ \sum_{i=1}^L\alpha_{ii}^2m_i^2+\sum_{1\leq i<j\leq L}(\alpha_{ij}m_i+\beta_{ij}m_j)^2}\\
&=&\sum_{i=1}^{L}\alpha_{ii}^{2}m_{i}^{2}+\sum_{ i<j}\alpha_{ij}^{2}m_{i}^{2}+\sum_{ i<j}\beta_{ij}^{2}m_{j}^2 + \sum_{ 1\leq i<j\leq L}
       2\alpha_{ij}\beta_{ij}m_{i}m_{j}\\
 &=&\sum_{i=1}^{L}(\alpha_{ii}^{2}+\alpha_{i,i+1}^{2}+\ldots+\alpha_{iL}^{2}+
        \beta_{1i}^2+\ldots\beta_{i-1,i}^2)m_{i}^2+\sum_{ 1\leq i<j\leq L}2\alpha_{ij}\beta_{ij}m_{i}m_{j}.
\end{eqnarray*}
	
	Now assuming, as we do, that $q(x)(x_{1}^2+\cdots + x_{n}^2)^r$ is a sobs,  by means of comparison of coefficients, we get a system
	of $|J_{1}|+|J_{2}|$ equations between reals. It is easily seen that these equations can be obtained as follows:
	For each $(i,\underline{i})\in J_{1}$ define
	\[T(i,\underline{i})=\{\text{indices}\  t\in \{1,\ldots,L\}\ \text{for which}\ m_{t}^2=
	x_{1}^{2i_{1}}\cdots x_{i}^{2i_{i}+2}\cdots  x_{n}^{2i_{n}} \},\]
	\[S(i,\underline{i})=\{\text{pairs}\  s_{1}<s_{2}\ \text{so that}\ m_{s_1}m_{s_2}=
	x_1^{2i_1}\cdots x_i^{2i_i+2}\cdots  x_n^{2i_n} \}\]
	and write the equation
	\[a_{i} {r\choose \underline i} =\sum_{ t\in T(i,\underline{i})} (\alpha_{tt}^{2}+\ldots+\alpha_{tL}^{2}+\beta_{1t}^{2}+\ldots+\beta_{t-1,t}^{2})+\sum_{(s_{1},s_{2})\in S(i,\underline{i})}2\alpha_{s_{1}s_{2}}\beta_{s_{1}s_{2}};\]
	for each $((i,j),\underline{i})\in J_{2}$, let
	\[S^{\prime}((i,j),\underline{i})=\{\text{pairs}\ s_{1}^{\prime} < s_{2}^{\prime}\ \text{ so that }\  m_{s_{1}^{\prime}}m_{s_{2}^{\prime}}=x_{1}^{2i_{1}}\ldots x_{i}^{2i_{i}+1}\ldots x_{j}^{2i_{j}+1}\ldots x_{n}^{2i_{n}}\},\]
	and write the equation
	\[d_{ij} {r\choose \underline i} =\sum_{ s_{1}^{\prime}, s_{2}^{\prime}\in S^{\prime}((i,j),\underline{i})}2\alpha_{s_{1}^{\prime} s_{2}^{\prime}}\beta_{s_{1}^{\prime} s_{2}^{\prime}}.\]
	Every system of reals
	$(\{a_{i}\}_{i=1}^n,\{d_{ij}\}_{1\leq i<j\leq n},\{\alpha_{ij}\}_{1\leq i\leq j\leq L},\{\beta_{ij}\}_{1\leq i < j\leq L})$ which satisfies the system of equations gives rise to a quadratic form $q$ and binomials so that $ q(x)(x_1^2+\cdots+x_n^2)^r$ is a sum of squares of these binomials. Now  if we have a system of reals satisfying the system, then we can find a particular new solution by replacing the $d_{ij}$ which are positive by $-d_{ij}$ and simultaneously replacing the
	$\beta_{s_1^{\prime} s_2^{\prime}}$ for which $s_1^\prime, s_2^\prime\in S^\prime((i,j),\underline{i})$ by $-\beta_{s_1^\prime s_2^\prime}$. Indeed note that the sets $S^\prime((i,j),\underline{i})$ are disjoint from the sets
	$S(i,\underline{i})$ and $(-\beta_{s_1^\prime s_2^\prime})^2=(\beta_{s_1^\prime s_2^\prime})^2$, hence the first set of $|J_{1}|$ equations will again be satisfied.
	What concerns the second set of equations we note that the sets $S^{\prime}((i,j),\underline{i})$ are also mutually disjoint, because a choice $(s_{1}^{\prime}, s_{2}^{\prime})$ defines via forming $ m_{s_{1}^{\prime}}m_{s_{2}^{\prime}}$ a unique power product $x_{1}^{2i_{1}}\ldots x_{i}^{2i_{i}+1}\ldots x_{j}^{2i_{j}+1}\ldots x_{n}^{2i_{n}}$ with exactly two odd exponents determining $i,j$ and then $\underline{i}$. In  other words $(s_{1}^{\prime}, s_{2}^{\prime})$ lives in only one of the sets $S^{\prime}((i,j),\underline{i})$ hence carrying through the replacements indicated we change the sign at the left hand side of an equation if and only if we change the sign of the corresponding right hand side. We therefore satisfy also the second group of equations.
	
	The new solution tells us that $\hat{q}(x) (x_1^2+\cdots+x_n^2)^r$
	is a sum of squares of binomials  where
	$\hat{q}(x)=\sum_{i=1}^{n}a_{i}x_{i}^{2}-\sum_{1\leq i<j\leq n} |d_{ij}|x_{i}x_{j}$. Now since the multiplier is evidently positive definite, $\hat{q}$ is nonnegative. Hence by Proposition~\ref{dmtsobsprop},  $q$ is a sum of squares of binomials.
\end{proof}

\section{{Quaternary quadratics and sums of trinomial squares}}\label{7}

The purpose of this section is to show that if a quarternary quadratic form $q(w,x,y,z)$ is not a sum of squares of trinomials then, given any positive integer $r$, the form $(w^2+x^2+y^2+z^2)^r\cdot q $ is not a sum of squares of trinomials. In fact it will be necessary to show more generally that for nonzero reals $\lambda_{1},\lambda_{2},\lambda_{3},\lambda_{4}$, the form  $(\lambda_{1}^2w^2+\lambda_{2}^2 x^2+\lambda_{3}^2 y^2+\lambda_{4}^2 z^2)^r\cdot q $ is not a sum of squares of trinomials.

\begin{prop}\label{scaledtrinomial} \label{thm:trinomialson4variables}
Let $q$ be a quaternary quadratic that is not a sum of trinomial squares.
Then the degree $(2r+2)$ form $p=(\lambda_1^2 w^2+\lambda_2^2 x^2+\lambda_3^2 y^2+\lambda_4^2 z^2)^r q$ is not a sum of trinomial squares, for any, not all zero, reals $\lambda_1,\lambda_2,\lambda_3,\lambda_4$.
\end{prop}
\begin{proof}
Let $\underline{x}=[w,x,y,z]{^T}$ and let $q= \underline{x}^T Q
	\underline{x}$. Since $Q$ is not in $FW_3^4$, there is an element $B$ in $(FW_3^4)^*$ such that $\langle Q, B\rangle <0.$ Moreover we can take that element to be in an extreme ray. If $B$ is psd, then $Q$ is not psd hence $q$ and therefore $p$ is not nonnegative and therefore not a sum of trinomial squares. So we just need to consider the case where $B$ is not psd, hence by Proposition~\ref{matrixform}, for some permutation $P$ and non singular matrix $D$  and some $a,c\in ]-\pi,\pi[\setminus\{0\}$ it has the following form
	\[{B^\prime}=DPBP^{T}D^{T}=     \begin{bmatrix}
	1   &\cos(a) & \cos(a-c)   &\cos(c)\\
	\cos(a)& 1  & \cos(c) & \cos(a- c)\\
	\cos(a-c)& \cos(c) &1       & \cos(a) \\
	\cos(c)& \cos(a- c) & \cos(a) &1
	\end{bmatrix} .\]
	We now have the inequality  $0> \langle  Q,B   \rangle= \langle  P^T D Q D P, {B^\prime}  \rangle. $
	We work with the new quadratic form $q_{\rm new}$  defined by
	 $q_{\rm new}=x^T P^T D Q D P x$ and show that given any  $ \lambda\in \R^4\setminus\{0\},$ we have that the associated quartic form
	 $p_{\rm new}=(\lambda_1^2 w^2+\lambda_2^2 x^2+\lambda_3^2 y^2+\lambda_4^2 z^2)^r q_{\rm new}$ is not a sum of trinomial squares.
	Since the property  `not being a sum of trinomial squares for any $\lambda$'  is invariant under permutations and scalings of the variables
	in $q_{\rm new}$, we shall get the claim concerning the original $p,q.$
	For simplicity of notation be aware that {we redefine}  ${\langle  Q,B   \rangle}:={\langle} P^T D Q D P, {B^\prime}{\rangle}$ and $(p,q):=(p_{\rm new},q_{\rm new}).$ The original $Q,B,p,q$
	will not play any further role in this proof.
	
	The polynomial $p$ is of degree $2r+2.$
	From Theorem~\ref{factork} we know that  $p$ has  a, usually nonunique, representation
	$p=z(x)_{r+1}^T Q' z(x)_{r+1},$  where $z(x)_{r+1}$ collects all monomials of degree $r+1$ and hence $Q'$ is an
	 ${r+4 \choose 3}\times  {r+4 \choose 3}$ matrix.
	We define the matrix $B'=(b_{ij}')$ as follows (where we use for the moment  as the most natural indexation, the one given by the vectors of exponents of the monomials),
	where $i,j\in \mathbb{Z}_{\geq 0}^4$ are {tuples} with $|i|=|j|=r+1$ so that $B'$ is also an  ${r+4 \choose 3}\times  {r+4 \choose 3}$ matrix:
	
	\[ b_{ij}'=    \left\{ \begin{array}{rl}
	b_{kl}  & \mbox{ iff $i+j$ has two odd entries exactly in positions $k\neq l$  }   \\
	1  & \mbox{ iff $i+j$ has only even entries }    \\
	0    & \mbox{ iff $i+j$ has 1 or 3 odd entries } \\
	\omega  & \mbox{ iff $i+j$ has only odd entries }
	\end{array}  \right . \]
	(The case that $i+j$ has exactly 1 or 3 odd entries can actually not happen in case $|i|=|j|$, but we will need the given rules below also in cases where $|i|\neq |j|$.)
	We will  show  that  $B'  \in (FW_3^{{r+4 \choose 3}})^*,$
	and then that $\langle B',Q'\rangle <0,$ thus showing  $Q'\not\in FW_3^{{r+4 \choose 3}},$ and hence
	showing by Proposition~\ref{factork} that $p$ is not a sum of squares of trinomials. We will then see from the  fact that being a sum of squares of trinomials is invariant
	under permutations,  that the original $p$ is also not a sum of squares of trinomials.

    To any string of exponents $i=(i_1,i_2,i_3,i_4)\in \ZZ_{\geq 0}^4$  we can associate a unique 4-{tuple} $\ve=\ve(i)=(\ve_1,\ve_2,\ve_3,\ve_4)\in \{0,1\}^4$ defined by $i_\nu \equiv \ve_\nu \mod 2.$

To prove  that $B'  \in (FW_3^{{r+4 \choose 3}})^*,$  note that its entries depend only on  $\ve(i+j)=\ve(i)+\ve(j)$ (computed in $\mathbb{Z}_2$).

If $|i|$ is even then the  only 4-{tuples} possible for  $\ve(i)$ are:

\centerline{ $0000,1100,1010,1001,0110,0101,0011,1111.$}
If $|i|$ is odd then  the only  4-{tuples} possible for $\ve(i)$ are:

\centerline{$1000,0100,0010,0001,1110,1101,1011,0111.$ }
 The  table below  is the modulo 2 addition table  for 4-{tuples} $\ve(i)$ with $| i |$ even
  (for example $1100+1001= 0101$).
  The reader verifies that precisely the same addition table would be obtained when the first line and the
 first column would be replaced by the 4-{tuples} $\ve(i)$ for which $|i|$ is odd.   If we replace the 4-{tuples} of the
 inner part of this table according to the rules given for the construction  of matrix  $B'$ we get the matrix
 that follows  the table.
  For example to 0101 corresponds $b_{24}.$  That matrix can serve as a look-up table for the construction
 of (sub)matrices of $B'.$
$$ \begin{array}{c|cccccccc}
        +      & 0000 &1100&1010&1001&0110&0101&0011&1111 \\ \hline
    0000  &  0000 &1100&1010&1001&0110&0101&0011&1111 \\
 1100     & 1100&0000&0110&0101&1010&1001&1111&0011 \\
1010     & 1010&0110&0000&0011&1100&1111&1001&0101\\
  1001   & 1001&0101 &0011&	0000&1111&1100&1010&0110\\
 0110     &0110&1010&1100 &1111&  0000&0011&0101&1001 \\
  0101     & 0101& 1001&1111& 1100&  0011 &  0000 &0110 &1010\\
  0011     &  0011& 1111  &1001 & 1010  & 0101 & 0110 & 0000& 1100\\
   1111     & 1111&  0011 & 0101&  0110 &  1001 & 1010 &1100 & 0000
  \end{array}  $$
   $$\begin{array}{cccccccc}
 1 &b_{12} &b_{13}  &b_{14}  &b_{23}  &b_{24}  &b_{34}  &\omega  \\
 b_{12}  &1 &b_{23}  &b_{24}  &b_{13}  &b_{14}  &\omega & b_{34} \\
 b_{13}  &b_{23}  &1 &b_{34}  &b_{12}  & \omega &b_{14}  &b_{24}  \\
 b_{14}  &b_{24}  &b_{34}  &1 & \omega &b_{12}  &b_{13}  &b_{23} \\
 b_{23}  &b_{13}  &b_{12}  &\omega &1 &b_{34}  &b_{24}  &b_{14}   \\
 b_{24}  &b_{14}  &\omega &b_{12}  &b_{34}  &1 &b_{23}  &b_{13}   \\
 b_{34}  &\omega &b_{14}  &b_{13}  &b_{24}  &b_{23}  &1 &b_{12}   \\
 \omega &b_{34}  &b_{24}  &b_{23}  &b_{14}  &b_{13}  &b_{12} & 1
 \end{array}$$

After having imposed some order on the set of $4$-{tuples} $i$  of 1-norm $|i|=1+r$ one can construct the matrix $B'.$
Consider now selecting three distinct 4-{tuples} $i,j,k$ of 1-norm $1+r$ and selecting in the matrix $B'$ the $3\times 3$
submatrix determined by this selection. If $i$ precedes $j$ precedes $k$ in the ordering of the 4-{tuples} the obtained
$3\times 3$ matrix is  the matrix at the left. Its entries are, as mentioned, completely determined by the matrix at the
right

\begin{center}
$\begin{pmatrix}
b_{ii}'  & b_{ij}' & b_{ik}' \\
 b_{ji}' & b_{jj}' & b_{jk}' \\
 b_{ki}' & b_{kj}' & b_{kk}'
\end{pmatrix}$ \hspace*{2cm}
$\begin{pmatrix}
\ve(i+i)  & \ve(i+j) & \ve(i+k)  \\
 \ve(j+i) & \ve(j+j) & \ve(j+k) \\
 \ve(k+i) & \ve(k+j) & \ve(k+k)
\end{pmatrix},$
\end{center}
from which it can be constructed using the above look-up table.  Hence the $3\times 3$ submatrix of $B'$ is simply permutation equivalent to a
principal $3\times 3$ submatrix and it is sufficient to show that  all principal $3\times 3$ submatrices of the look-up table  are {psd}. To see this  note first that the left upper
 $4\times 4$ matrix of the look-up table  coincides with $B.$ More generally all principal
  $3\times 3$ submatrices   of the look up table
  which do not contain an $\omega$  are permutation equivalent to $3\times 3$ principal
  submatrices of $B$ and hence are automatically {psd}. The $3\times 3$ principal submatrices
  containing  $\omega$ stem from selecting sets of three line indices which contain one of the
  sets $\{1,8\},\{2,7\},\{3,6\},\{4,5\}.$
  These matrices are permutation equivalent to one of the following matrices:

		\[\begin{bmatrix}
		1&\omega & b_{12}\\
		\omega& 1& b_{34} \\
		b_{12}&b_{34}&1 	
		\end{bmatrix},\hy
		\begin{bmatrix}
		1&\omega & b_{14}\\
		\omega& 1& b_{23} \\
		b_{14}&b_{23}&1 	
		\end{bmatrix},\hy
		\begin{bmatrix}
		1&\omega & b_{13}\\
		\omega& 1& b_{24} \\
		b_{13}&b_{24}&1 	
		\end{bmatrix}.\]

So it is sufficient  to find an  $\omega\in \mathbb{R}$  such that these  matrices are  {psd}.    To see this, the easiest choice is  to put $\omega=1.$ This is a universal choice valid for all $0<a,c<\pi$ that
  result in determinants equal to 0. If one is given
   explicit real numbers for $a,b,c,$  then putting  $\omega=(1-\ve)$ for sufficiently small $\ve>0,$
    one will obtain strictly positive definite (sub)determinants.  With these checks we have proved that
     $B'  \in (FW_3^{{r+4 \choose 3}})^*.$

   We now show the other claim we made for $B'.$ \vs \vs

Claim:  There holds $\langle B',Q'\rangle = (\sum_{i=1}^4 \lambda_i^2 )^r \, \langle B,Q\rangle.$ Thus
   $\langle Q', B' \rangle <0.$ \vs \vs

  By the definition of the inner product in matrix space, we have to show

   \vs \centerline{ $\ds  \sum \{b_{ij}' q_{ij}': \; i,j\in \mathbb Z_{\geq 0}^4,  |i|=|j|=1+r  \}
         =(\sum_{i=1}^4 \lambda_i^2)^r  \sum_{i,j=1}^4 b_{ij}q_{ij}.$    }

Now, given  $ i,j\in \mathbb Z_{\geq 0}^4,  |i|=|j|=1+r,$  we have of course $|i+j|=2r+2.$ Furthermore for any such sum $s=i+j$  we  have {a priori} exactly one of the following
possibilities:     all entries are even;     exactly two entries are odd;   one or three entries are odd;     all entries are odd.

Since for  an  $s\in \mathbb Z_{\geq 0}^4$  for which $|s|$ is even it is, as noted already,  impossible
that $s$ has exactly one or three odd entries,  we can write the left side above as follows:

\vs\vs\centerline{
$\ds \sum_{\scriptsize \begin{array}{c}
 |s|=2r+2\\ s \text{ has four}\\ \text{even entries}   \end{array}}
  \sum_{\scriptsize \begin{array}{c}
 |i|=|j|=r+1 \\ i+j=s  \end{array}} \ny b_{ij}' q_{ij}'
+ \ny
\sum_{\scriptsize \begin{array}{c}
 |s|=2r+2 \\ s  \text{ has two} \\\text{ odd entries}   \end{array}}
  \sum_{\scriptsize \begin{array}{c}
 |i|=|j|=r+1\\ i+j=s  \end{array}} \ny  b_{ij}' q_{ij}'  +
\ny   \sum_{\scriptsize \begin{array}{c}
 |s|=2r+2\\ s \text{ has four}\\ \text{odd entries}   \end{array}}
  \sum_{\scriptsize \begin{array}{c}
 |i|=|j|=r+1 \\ i+j=s  \end{array}} \ny\ny b_{ij}' q_{ij}' .$}

\vs By the definition of $B'$ given, this is equal to

\vs\vs\centerline{
$\ds  \ny\ny
\sum_{\scriptsize \begin{array}{c}
 |s|=2r+2 \\ s \text{ has four } \\ \text{even entries}  \end{array}}
\ny  \sum_{\scriptsize \begin{array}{c}
 |i|=|j|=r+1\\ i+j=s  \end{array}} \ny\ny  q_{ij}' +
 \sum_{1\leq k<l\leq 4}
 \sum_{\scriptsize \begin{array}{c}
 |s|=2r+2 \\ s \text{ has  odd}\\ \text{entries at } k,l   \end{array}}
\ny\ny  \sum_{\scriptsize \begin{array}{c}
 |i|=|j|=r+1 \\ i+j=s  \end{array}}  \ny\ny b_{kl} q_{ij}'
 +\ny
  \sum_{\scriptsize \begin{array}{c}
 |s|=2r+2\\ s \text{ has four}\\ \text{odd entries}   \end{array}}
\ny  \sum_{\scriptsize \begin{array}{c}
 |i|=|j|=r+1 \\ i+j=s  \end{array}} \ny\ny \omega  q_{ij}' .$}

\vs Now we remember that by its construction, polynomial $p$ cannot have a monomial with only odd exponents so the third sum is 0.
 The sum of the coefficients of monomials whose variables  have only even powers  in $p$ is given by  Lemma~\ref{coefficients}   by

\vs \centerline{$\ds(\sum_{i=1}^4 \lambda_i^2 )^r (q_{11}+q_{22}+q_{33}+q_{44});$}

 while the second sum is

\vs \vs \centerline{$ \ds \sum_{1\leq k<l\leq 4} b_{kl}
 \sum_{\scriptsize \begin{array}{c}
 |s|=2r+2\\ s \text{ has  odd}\\\text{ entries at } k,l   \end{array}}
  \sum_{\scriptsize \begin{array}{c}
 |i|=|j|=r+1\\ i+j=s  \end{array}}   q_{ij}' $}

\vs The inner double sum here
 can  be described exactly as the sum of the coefficients of the monomials of $p$ which have two odd entries
 at distinct $k,l.$  Hence again by Lemma~\ref{coefficients}  the  inner double sum is equal to   $2(\sum \lambda_i^2 )^r q_{kl}$ and so the
 sum is

\vs \centerline{ $2 \ds  ( \sum_{i=1}^4 \lambda_i^2)^r  \sum_{1\leq k<l\leq 4}  b_{kl} q_{kl} =
( \sum_{i=1}^4 \lambda_i^2)^r  \sum_{\scriptsize \begin{array}{c} 1\leq k,l \leq 4,\\ k\neq l   \end{array} }  b_{kl} q_{kl}.$ }
 The  claim  now follows because      $\sum_{i=1}^4 \lambda_i^2>0.$   \vs \vs

To conclude the proof we detail an idea we mentioned at the beginning. We have till now shown that  whatever the reals
 $\lambda_1,...,\lambda_4,$  (not all zeros) are, if  the polynomial
$q_{\rm new}=x{^T} P'{^T} Q P' x,$ (with $Q$ satisfying the hypotheses)  then the polynomial
$p_{\rm new}=(\lambda_1^2 w^2+\lambda_2^2 x^2+\lambda_3^2 y^2+\lambda_4^2 z^2)^r q_{\rm new}$
is not sum of trinomial squares.
Now by its definition $q_{\rm new}(w,x,y,z)=q(\pi(w),\pi(x),\pi(y),\pi(z))$  where $\pi$ embodies the permutation matrix $P'.$
Since the property `to be a sum of squares of trinomials'  is evidently invariant under permutations,
it follows  that $(\lambda_1^2 \pi^{-1}(w)^2+\lambda_2^2 \pi^{-1}(x)^2+\lambda_3^2 \pi^{-1}(y)^2+\lambda_4^2 \pi^{-1}(z)^2)^r q(w,x,y,z)$
is  not a sum of trinomial squares for any $\lambda_1,...,\lambda_4$.
  Since  $\{\pi^{-1}(w), \pi^{-1}(x), \pi^{-1}(y),\linebreak \pi^{-1}(z)\}=\{w,x,y,z\}$
it follows that  $(\lambda_1^2 w^2+\lambda_2^2 x^2+\lambda_3^2 y^2+\lambda_4^2 z^2)^r q(w,x,y,z)$ is not a sum of
 trinomial squares.
    \end{proof}
  	
In particular by taking $\lambda_1=\lambda_2=\lambda_3=\lambda_4=1$ we have that a quaternary quadratic is $r$-so$3$s if and only if it is a so$3$s, completing the proof of our main Theorem \ref{prop:main_result}.
%

\section{Quinary quadratics and sums of tetranomial squares}\label{8}

Up to now we established three results (Theorems  \ref{thm:generalAhmadiMajumdar}, \ref{Quadraticsobs} and \ref{thm:trinomialson4variables}) that show that
quadratics on $n$ variables are $r$-so$k$s if and only if they are so$k$s under certain assumptions, namely that they are symmetric, that $k=2$ or that $n \leq 4$.
A natural belief that may occur to the reader is that in fact the same would hold without such assumptions. In this section we give a counterexample to that  natural conjecture.
We give  a quadratic form in $5$ variables which is not so$4$s but that becomes  so$4$s after multiplication with $x_1^2+\cdots+x_5^2$.

\begin{example}
{\em Consider the matrix $M$ given by
$$M=\left[
\begin{array}{lllll}
49 & -21 & 37 & -37 & -21 \\
-21 & 17 & -21 & 21 & 29 \\
37 & -21 & 41 & -25 & -33 \\
-37 & 21 & -25 & 41 & 33 \\
-21 & 29 & -33 & 33 & 73
\end{array}
\right].$$
This matrix is not in $FW_4^5$. To see this just verify that the matrix
\[A=\begin{bmatrix}
3&1& -2& 2& -1\\
1& 3& 0& 0& -1\\
-2& 0& 2& -1& 1\\
2& 0& -1& 2& -1\\
-1& -1& 1& -1& 1
\end{bmatrix}\]
is in $(FW_4^5)^*$, by checking that all its $4 \times 4$ principal submatrices are psd, and note that $\langle A, M \rangle = -1 < 0$.

Consider the quadratic form $q_M = x^T M x$. By our previous observation, $q_M$ is not so$4$s. Let then $p_M=	(x_1^2+x_2^2+x_3^2+x_4^2+x_5^2)\cdot q_M $. We claim that $p_M$ is so$4$s, hence, $q_M$ is $1$-so$4$s. To prove it one would have to provide an exact certificate. One can easily check that $p_M=z(x)_2^T Q z(x)_2$ where \begin{tiny}
	$$\ny\ny\ny\ny Q=
	\begin{bmatrix}
	49  & -21   &  0  &  37   &  0    &  0    & -37   &   0   & -5  &   0   & -21  &  0   &  0   &  0   &  0   \\
	-21  &  66   & -21 & -21   &  37   & -11/5 &  21   &  -37  &  0  & -17/5 &  29  & -21  &  0   &  0   &  0   \\
	0  & -21   &  17 &   0   & -21   &  0    &  0    &   21  &  0  &   0   &   0  &  29  &  0   &  0   &  0   \\
	37  & -21   &  0  &  90   & -94/5 &  37   & -20   &   0   & -37 &   0   & -33  &  0   & -14  &  0   &  0   \\
	0  &  37   & -21 & -94/5 &  58   & -21   &  0    &  -25  &  21 &   0   &   0  & -33  &  29  &  0   & -4   \\
	0  & -11/5 &  0  &  37   & -21   &  41   &  0    &   0   & -25 &   0   &  -7  &  0   & -33  &  0   &  0   \\
	-37 &  21   &  0  & -20   &  0    &  0    &  90   & -88/5 &  37 &  -37  &  33  &  0   &  0   &  12  &  0   \\
	0  & -37   &  21 &   0   & -25   &  0    & -88/5 &   58  & -21 &   21  &   0  &  33  &  0   &  29  & 17/5 \\
	-5  &  0    &  0  & -37   &  21   & -25   &  37   &  -21  &  82 &  -25  &   0  &  0   &  33  & -33  & -23/5 \\
	0  & -17/5 &  0  &   0   &  0    &  0    &  -37  &   21  & -25 &   41  &  -9  &  0   &  0   &  33  &  0    \\
	-21 &  29   &  0  & -33   &  0    &  -7   &  33   &   0   &  0  &  -9   & 122  & -21  &  37  & -37  & -21   \\
	0  & -21   &  29 &   0   & -33   &  0    &  0    &   33  &  0  &   0   & -21  &  90  & -17  & 88/5 &  29   \\
	0  &  0    &  0  &  -14  &  29   & -33   &  0    &   0   &  33 &   0   &  37  & -17  &  114 &-102/5& -33\\
         0  &  0    &  0  &   0   &  0    &  0    &  -12  &   29  & -33 &  33   & -37  & 88/5 &-102/5& 114  &  33  \\
	0  &  0    &  0  &   0   &  -4   &  0    &   0   &  17/5 &-23/5&   0   & -21  &  29  & -33  &  33  &  73
	\end{bmatrix}. $$
\end{tiny}
It remains to show that this matrix is in fact in $FW_4^{15}$. In general, such matrices are sums of up to $\binom{15}{4}=1365$ matrices with $4\times 4$ support, and generating rational decompositions is certainly not trivial. In this case the example was chosen in such a way that numerically we can do it using only $27$ such matrices (in fact possibly all with rank one) with supports  $K\times K$ with $K$ as follows; we write $1,2,4,7$ instead of
$\{1,2,4,7\},$ etc.:
\begin{center}
	\begin{tabular}{llllll}
		1,2,4,7   &  $1,2,4,11$  &  $1,2,7,11$   & $1,4,7,9$    &  $2,3,5,8$   &  $2,3,5,12$  \\   $ 2,3,8,12$ &
		2,4,5,6   &  $2,5,8,12$  &  $2,7,8,10$   & $3,5,8,12$   & $4,5,6,9$  \\    $4,5,6,13$   &  $4,5,9,13$  &
		4,6,11,13 &  $5,6,9,13$  & $5,12,13,15$ & $7,8,9,10$  \\ $ 7,8,9,14$   &  $7,10,11,14$ & $8,9,10,14$ & $8,12,14,15$ &  $9,13,14,15$ &  $11,12,13,14$ \\ $11,12,13,15$ & $11,12,14,15$ & $11,13,14,15$ &  &   &    \\
	\end{tabular}
\end{center}
Since to put the  27 matrices with their floating point entries  themselves at this place would be too space consuming,  the reader interested
to check the example can  obtain them by request from the first author.

We did simply a numerical verification, but due to the small size of the calculation we have confidence in the example. Further work would involve
 rationalizing this certificate, in order to eliminate any remaining doubts.}
\end{example}


 \bibliographystyle{elsarticle-num}
 \bibliography{bibliography}

\begin{thebibliography}{10}
\expandafter\ifx\csname url\endcsname\relax
  \def\url#1{\texttt{#1}}\fi
\expandafter\ifx\csname urlprefix\endcsname\relax\def\urlprefix{URL }\fi
\expandafter\ifx\csname href\endcsname\relax
  \def\href#1#2{#2} \def\path#1{#1}\fi

\bibitem{Ahm14}
A.~A. Ahmadi, A.~Majumdar, D{SOS} and {SDSOS} optimization: more tractable
  alternatives to sum of squares and semidefinite optimization, SIAM J. Appl.
  Algebra Geom. 3~(2) (2019) 193--230.
\newblock \href {http://dx.doi.org/10.1137/18M118935X}
  {\path{doi:10.1137/18M118935X}}.

\bibitem{reznick1987}
B.~Reznick, A quantitative version of {H}urwitz' theorem on the
  arithmetic-geometric inequality, J. Reine Angew. Math. 377 (1987) 108--112.
\newblock \href {http://dx.doi.org/10.1515/crll.1987.377.108}
  {\path{doi:10.1515/crll.1987.377.108}}.

\bibitem{reznick1989}
B.~Reznick, Forms derived from the arithmetic-geometric inequality, Math. Ann.
  283~(3) (1989) 431--464.
\newblock \href {http://dx.doi.org/10.1007/BF01442738}
  {\path{doi:10.1007/BF01442738}}.

\bibitem{Hurwitz}
A.~Hurwitz, Ueber den {V}ergleich des arithmetischen und des geometrischen
  {M}ittels, J. Reine Angew. Math. 108 (1891) 266--268.
\newblock \href {http://dx.doi.org/10.1515/crll.1891.108.266}
  {\path{doi:10.1515/crll.1891.108.266}}.

\bibitem{hilbert1888}
D.~Hilbert, {\"U}ber die {D}arstellung definiter {F}ormen als {S}umme von
  {F}ormenquadraten, Mathematische Annalen 32~(3) (1888) 342--350.

\bibitem{Reznick1995}
B.~Reznick, \href{https://doi.org/10.1007/BF02572604}{Uniform denominators in
  {H}ilbert's seventeenth problem}, Math. Z. 220~(1) (1995) 75--97.
\newblock \href {http://dx.doi.org/10.1007/BF02572604}
  {\path{doi:10.1007/BF02572604}}.
\newline\urlprefix\url{https://doi.org/10.1007/BF02572604}

\bibitem{Lasserre2001}
J.~B. Lasserre, Global optimization with polynomials and the problem of
  moments, SIAM Journal on optimization 11~(3) (2001) 796--817.

\bibitem{parrilo2000}
P.~A. Parrilo, Structured semidefinite programs and semialgebraic geometry
  methods in robustness and optimization, Ph.D. thesis, California Institute of
  Technology (2000).

\bibitem{Johnson}
R.~A. Horn, C.~R. Johnson, Matrix analysis, 2nd Edition, Cambridge University
  Press, Cambridge, 2013.

\bibitem{Ramana95}
M.~Ramana, A.~J. Goldman, Some geometric results in semidefinite programming,
  J. Global Optim. 7~(1) (1995) 33--50.
\newblock \href {http://dx.doi.org/10.1007/BF01100204}
  {\path{doi:10.1007/BF01100204}}.

\bibitem{SDOCAG}
G.~Blekherman, P.~A. Parrilo, R.~R. Thomas (Eds.), Semidefinite optimization
  and convex algebraic geometry, Vol.~13 of MOS-SIAM Series on Optimization,
  Society for Industrial and Applied Mathematics (SIAM), Philadelphia, PA;
  Mathematical Optimization Society, Philadelphia, PA, 2013.

\bibitem{boman2005factor}
E.~G. Boman, D.~Chen, O.~Parekh, S.~Toledo, On factor width and symmetric
  {$H$}-matrices, Linear Algebra Appl. 405 (2005) 239--248.
\newblock \href {http://dx.doi.org/10.1016/j.laa.2005.03.029}
  {\path{doi:10.1016/j.laa.2005.03.029}}.

\bibitem{Permenter2017}
F.~Permenter, P.~Parrilo, Partial facial reduction: simplified, equivalent
  {SDP}s via approximations of the {PSD} cone, Math. Program. 171~(1-2, Ser. A)
  (2018) 1--54.
\newblock \href {http://dx.doi.org/10.1007/s10107-017-1169-9}
  {\path{doi:10.1007/s10107-017-1169-9}}.

\bibitem{bernstein2009matrix}
D.~S. Bernstein, Matrix Mathematics: Theory, Facts, and Formulas (Second
  Edition), Princeton University Press, 2009.

\bibitem{muir1933treatise}
T.~Muir, A treatise on the theory of determinants, Revised and enlarged by
  William H. Metzler, Dover Publications, Inc., New York, 1960.

\bibitem{Fidalgo2011}
C.~Fidalgo, A.~Kovacec, Positive semidefinite diagonal minus tail forms are
  sums of squares, Math. Z. 269~(3-4) (2011) 629--645.
\newblock \href {http://dx.doi.org/10.1007/s00209-010-0753-y}
  {\path{doi:10.1007/s00209-010-0753-y}}.

\end{thebibliography}

\end{document}